\documentclass[letterpaper]{article}

\usepackage[letterpaper,left=1.5in,right=1.5in,top=1.5in,bottom=1.5in]{geometry}

\newcommand{\myauthor}{Benjamin Antieau and Greg Stevenson}
\newcommand{\mytitle}{Derived categories of representations of small categories over commutative noetherian rings}
\newcommand{\pdftitle}{\mytitle}
\author{\myauthor}
\date{}

\title{\mytitle}

\usepackage[pdfstartview=FitH,
            pdfauthor={\myauthor},
            pdftitle={\pdftitle},
            colorlinks,
            linkcolor=reference,
            citecolor=citation,
            urlcolor=e-mail,
            backref]{hyperref}

\usepackage{mathrsfs}
\usepackage[mathscr]{euscript}
\usepackage{amsmath}
\usepackage{amscd}
\usepackage{amsbsy}
\usepackage{amssymb}
\usepackage{microtype}
\usepackage{enumerate}
\usepackage[all,cmtip]{xy}
\usepackage{tikz}
\usetikzlibrary{matrix,arrows}
\usepackage{dsfont}
\usepackage[abbrev,lite]{amsrefs}
\usepackage{amsthm}

\usepackage{color}
\definecolor{todo}{rgb}{1,0,0}
\definecolor{conditional}{rgb}{0,1,0}
\definecolor{e-mail}{rgb}{0,.40,.80}
\definecolor{reference}{rgb}{.20,.60,.22}
\definecolor{mrnumber}{rgb}{.80,.40,0}
\definecolor{citation}{rgb}{0,.40,.80}

\newtheorem*{theorem*}{Theorem}

\theoremstyle{definition}
\newtheorem*{ack}{Acknowledgements}{}

\newcommand{\prn}[1]{\{#1\}}
\newcommand{\rprn}[1]{\textcolor{todo}{[}#1\textcolor{todo}{]}}

\newcommand{\perf}{\mathrm{perf}}

\DeclareMathOperator{\Modu}{Mod}
\DeclareMathOperator{\Mod}{Mod}

\newcommand{\we}{\simeq}
\newcommand{\iso}{\cong}

\newcommand{\stovicek}{\v{S}\v{t}ov{\'\i}\v{c}ek}

\newcommand{\add}{\mathrm{add}}
\newcommand{\qc}{\mathrm{qc}}

\newcommand{\opp}{\mathrm{op}}

\DeclareMathOperator{\Spec}{Spec}

\DeclareMathOperator{\Ext}{Ext}

\newcommand{\Hom}{\mathrm{Hom}}
\newcommand{\Fun}{\mathrm{Fun}}

\newcommand{\Lrm}{\mathrm{L}}

\newcommand{\Mrm}{\mathrm{M}}

\newcommand{\Drm}{\mathrm{D}}

\newcommand{\Erm}{\mathrm{E}}

\newcommand{\Srm}{\mathrm{S}}
\newcommand{\Trm}{\mathrm{T}}

\newcommand{\Lscr}{\mathscr{L}}

\newcommand{\CC}{\mathds{C}}

\newcommand{\ZZ}{\mathds{Z}}

\newcommand{\PP}{\mathds{P}}

\theoremstyle{plain}
\newtheorem{theorem}{Theorem}[section]
\newtheorem{lemma}[theorem]{Lemma}
\newtheorem{proposition}[theorem]{Proposition}

\newtheorem{corollary}[theorem]{Corollary}

\theoremstyle{definition}
\newtheorem{definition}[theorem]{Definition}

\newtheorem{example}[theorem]{Example}
\newtheorem{question}[theorem]{Question}

\newtheorem{remark}[theorem]{Remark}

\let\oldmarginpar\marginpar
\renewcommand\marginpar[1]{\-\oldmarginpar[\raggedleft\footnotesize #1]%
{\raggedright\footnotesize #1}}

\usepackage{txfonts}

\begin{document}
\maketitle

\begin{abstract}

    \noindent
		We study the derived categories of small categories over commutative noetherian 
		rings. Our main result is a parametrization of the localizing subcategories in 
		terms of the spectrum of the ring and the localizing subcategories over residue 
		fields. In the special case of representations of Dynkin quivers over a commutative 
		noetherian ring we give a complete description of the localizing subcategories 
		of the derived category, a complete description of the thick subcategories 
		of the perfect complexes and show the telescope conjecture holds. 
		We also present some results concerning the telescope conjecture more generally.
		

    \paragraph{Mathematics Subject Classification 2010}
    Primary:
    \href{http://www.ams.org/mathscinet/msc/msc2010.html?t=16Exx&btn=Current}{16E35},
    \href{http://www.ams.org/mathscinet/msc/msc2010.html?t=16Gxx&btn=Current}{16G20}.
    Secondary:
    \href{http://www.ams.org/mathscinet/msc/msc2010.html?t=13Dxx&btn=Current}{13D09},
    \href{http://www.ams.org/mathscinet/msc/msc2010.html?t=18Gxx&btn=Current}{18G55}.

\end{abstract}

\section{Introduction}
\let\thefootnote\relax\footnote{The first author was supported by NSF Grant No. DMS-1461847.
The second author was partly supported by NSF Grant No.\
0932078 000 while in residence at the Mathematical Science Research Institute and by the
Alexander von Humboldt Stiftung.}

If $\mathrm{T}$ is a triangulated category with all coproducts, a localizing subcategory
$\Lrm\subseteq\mathrm{T}$ is a full triangulated subcategory closed under all coproducts in
$\mathrm{T}$. Localizing subcategories are so-named because in good cases
(the Bousfield localizations) the Verdier quotient functor $\Trm\rightarrow \Trm/\Lrm$ possesses a right
adjoint, i.e.\ they give rise to localization functors. Understanding the collection of localizing subcategories 
on a given triangulated category is a challenging and interesting problem which has been completely resolved 
in only a few classes of examples. 

The history of such problems has roots in stable homotopy theory, where one would like to
relate two localizations of the $p$-local stable homotopy category $\mathrm{SH}_{(p)}$: one which
has excellent theoretical properties (localization with respect to the homology theory given
by the Johnson-Wilson spectrum $\Erm(n)$) and one which is computable (the telescopic
localization).  The importance of such questions arose first in the work of
Bousfield~\cite{bousfield-localization} and
Ravenel~\cite{ravenel-localization}.
That these two localizations agree is the still-open telescope
conjecture. Work on nilpotence closely related to the telescope conjecture by
Devinatz, Hopkins, and Smith~\cite{devinatz-hopkins-smith,hopkins-smith}
has led to the classification of all thick subcategories, i.e.\ triangulated subcategories closed under direct summands, of $\mathrm{SH}^{\mathrm{fin}}$, the
homotopy category of finite spectra. Using similar ideas on the detection of nilpotent maps
between objects in $\Drm(R)$, Neeman~\cite{neeman-chromatic} classified the localizing
subcategories of $\Drm(R)$ and the thick subcategories of $\Drm^\perf(R)$ when $R$ is
noetherian in terms of $\Spec R$.

Going beyond the example of $\Drm(R)$ where $R$ is noetherian and commutative seems rather
difficult. In terms of classification of thick subcategories of $\Drm^\perf(X)$, when $X$
is a quasi-compact and quasi-separated scheme, one has the result of
Thomason~\cite{thomason-triangulated}, which says that the thick subcategories which are also tensor
ideals correspond bijectively to unions of closed subsets of $X$ with quasi-compact
complement. This kind of result has been taken up by other authors, such as
Benson-Carlson-Rickard~\cite{benson-carlson-rickard} and Benson-Iyengar-Krause~\cite{benson-iyengar-krause},
who study the tensor ideals of stable module categories of finite groups. This is part of a
generalized framework of studying tensor ideals, pursued by Balmer~\cite{balmer},
Dell'Ambrogio-Stevenson~\cite{dellambrogio-stevenson-graded,dellambrogio-stevenson-spectra},
and Stevenson~\cite{stevenson-support,stevenson-singularity}.

In contrast to all that is known about thick subcategories, very little is known about
localizing subcategories outside of Neeman's theorem. For instance, one does not know all
localizing subcategories of $\Drm_{\qc}(\PP^1_\CC)$. We mention one more example,
due to Br\"uning~\cite{bruning}, who classified the localizing subcategories
of $\Drm(A)$ where $A$ is a hereditary Artin algebra of finite representation type.

Let $R$ be a noetherian commutative ring.
We show that in many cases classification of the localizing subcategories of an $R$-linear triangulated category can be
reduced to to studying the localizing subcategories of the ``fibers'' over the residue
fields of $R$.

Let $C$ be a small category,
and let $s:\Lscr\rightarrow\Spec R$ denote the class constructed fiber by fiber over $\Spec R$,
by letting $s^{-1}(p)$, for $p\in \Spec R$, be the class of localizing subcategories of $\Drm(k(p)C)$.
Note that, a priori, the localizing subcategories of $\Drm(k(p)C)$ only form a
proper class, which is the reason for the careful wording above. There is, however, no known example of a compactly 
generated triangulated category whose collection of localizing subcategories does not form a set.
 The following result is our first theorem.

\begin{theorem*}[\ref{cor:main}]\label{thm:classification}
    Let $R$ be a noetherian commutative ring and $C$ a small category. Then
    there is an isomorphism of lattices
    \begin{equation*}
        \xymatrix{
            \left\{\text{localizing subcategories $\Lrm$ of $\Drm(RC)$}\right\} \ar@<1ex>[r]^-f    &
            \left\{\text{sections $l$ of $\Lscr\xrightarrow{s}\Spec R$}\right\}\ar@<1ex>[l]^-g
        },
    \end{equation*}
    where $f$ takes a localizing subcategory $\Lrm$ of $\Drm(RC)$ to the function $l:\Spec
    R\rightarrow\Lscr$ such that $l(p)=\add(k(p)\otimes_R\Lrm)$, and where $g(l)$ is the
    localizing subcategory generated by all $X$ such that $k(p)\otimes_RX\in l(p)$ for all
    $p\in\Spec R$.
\end{theorem*}

In fact, our methods apply somewhat more generally, allowing one to replace $\Drm(RC)$ with derived
categories of representations of $R$-flat $R$-linear categories.

Our second result is a classification of the telescopic localizations of $\Drm(RQ)$ and a
classification of the thick subcategories of $\Drm^\perf(RQ)$ when $Q$ is a Dynkin quiver.

\begin{theorem*}[\ref{cor:ADEclassification}, \ref{cor:ADE1}, \ref{cor:ADE2}]
    Let $R$ be a noetherian commutative ring, $Q$ a simply laced Dynkin quiver, and denote by $RQ$ the $R$-linear path algebra of $Q$. There is an 
    isomorphism of lattices
    \begin{displaymath}
    \xymatrix{
            \left\{\text{localizing subcategories of $\Drm(RQ)$}\right\} \ar@<1ex>[r]^-f    &
            \left\{\text{functions $\Spec R \rightarrow \mathrm{NC}(Q)$}\right\}\ar@<1ex>[l]^-g
        },
    \end{displaymath}
    where $\mathrm{NC}(Q)$ denotes the lattice of noncrossing partitions associated to $Q$. 
		
    Moreover, the 
    telescope conjecture holds for $\Drm(RQ)$ and the smashing subcategories, which by virtue of the telescope conjecture are in bijection with thick subcategories of 
    $\Drm^\perf(RQ)$, correspond to those $\sigma\colon \Spec R \to \mathrm{NC}(Q)$ such that whenever $p\subseteq q$ in 
    $\Spec R$ we have $\sigma(p)\leq \sigma(q)$.
\end{theorem*}

In terms of the localizing subcategories, this theorem basically combines
Theorem~\ref{thm:classification} with the results of Ingalls and
Thomas~\cite{ingalls-thomas} on localizing subcategories of $\Drm(kQ)$ for fields $k$.

Initially, we had also hoped to prove the telescope conjecture for the telescopic localizations
of $\Drm(RC)$ more generally, at least with some hopefully mild hypothesis. This turned out to be overly ambitious, but 
we present some partial results in Section~\ref{sec:telescope}.

\begin{ack}
We would like to express our thanks, to MSRI for its hospitality during the thematic program ``Noncommutative algebraic
geometry and representation theory,'' and to
Universit\"at Bielefeld for its hospitality to the first author. We are also grateful to the anonymous referee for providing several helpful comments.
\end{ack}

\section{Preliminaries on representations of small categories}

Throughout we fix a commutative ring $R$. Let $C$ be a small category.

\begin{definition}
    The category of \emph{right} $C$\emph{-modules over} $R$ is the functor category
    \begin{displaymath}
        \Modu_R C = \Fun(C^\opp, \Modu R)
    \end{displaymath}
    consisting of contravariant functors from $C$ to the category of $R$-modules.
\end{definition}

The following well known lemma ensures that we can use the standard tools of
homological algebra when dealing with $C$-modules.

\begin{lemma}
The category $\Modu_R C$ of right $C$-modules over $R$ is a Grothendieck category with enough projectives.
\end{lemma}

\begin{proof}
    Recall that a Grothendieck (abelian) category is an abelian category (1) satisfying axiom
    (AB5), on the existence and exactness of filtered colimits, and (2)
    possessing a generator. The lemma can be proved by showing that the direct
    sum of the set of representable objects is a generator, that filtered
    colimits are computed pointwise so that (AB5) follows from the satisfaction
    of that axiom for $\Modu_R$ itself, and finally that the projective objects
    of $\Modu_R C$ are summands of direct sums of representables. Details are
    left to the reader.
\end{proof}

We can also approach $C$-modules via $R$-linear functors.

\begin{definition}
    The \emph{$R$-linearization} of $C$, which we will denote by $RC$,
    is the category with the same objects as $C$ and 
    whose hom-objects are free $R$-modules on the hom-sets of $C$
    \begin{displaymath}
    RC(c,c') = \bigoplus_{f\in C(c,c')} Rf,
    \end{displaymath}
    with the obvious composition rule. In other words, $RC$ is the free $R$-linear category on $C$.
\end{definition}

\begin{definition}
    An $R$-linear category $D$ is a small category enriched in $R$-modules. It is
    \emph{flat} if $D(c,c')$ is a flat $R$-module for all pairs of objects $c,c'$ in $D$.
\end{definition}

\begin{definition}
    If $D$ is an $R$-linear category, then the category of right $D$-modules over $R$
    is defined to be the functor category
    \begin{equation*}
        \Mod_R D=\Fun_R(D,\Modu R)
    \end{equation*}
    of $R$-linear functors.
\end{definition}

Evidently, $RC$ is a flat $R$-linear category for any small category $C$, since the hom-objects are
free. The reason for looking at these more general categories is to capture the
representation theory of $R$-algebras ``with many objects'', whereas the representations of $RC$ are
representations of monoids with many objects. In the case where $C$ has one
object with monoid of endomorphisms $M$, the category of representations of $C$
in $R$-modules is equivalent to the category of right $R[M]$-modules, where $R[M]$ is the
monoid algebra of $M$. On the other hand, if $D$ is an $R$-linear category with
one object having endomorphism algebra $S$, then $S$ is an $R$-algebra, and the
category of $R$-linear representations of $D$ is equivalent to the category of
right $S$-modules. Of course, not every $R$-algebra is a monoid algebra, so the
$R$-linear categories capture more examples.

Of course, we should now check that $\Modu_R C$ and $\Modu_R RC$ are equivalent. We do this in
a moment, but we first want to introduce extra structure that will be preserved.
Tensoring an $RC$-module objectwise with an $R$-module defines a bifunctor
\begin{displaymath}
    \Modu R \times \Mod_R RC \stackrel{\otimes_R}{\rightarrow} \Mod_R RC
\end{displaymath}
which is explicitly given by $(M\otimes_RF)(c) = M\otimes_RF(c)$ for an 
$R$-module $M$, $RC$-module $F$, and $c\in C$. This gives an action of 
the category of $R$-modules on the category of $RC$-modules. We note this action is nothing 
other than the existence of copowers for the $R$-linear category $\Mod_R RC$. There is, of course, a similar
action on $\Mod_R D$ when $D$ is an $R$-linear category.

\begin{remark}\label{rem:lazy}
	Here and in the sequel we will work with categories of the form $RC$ since our main examples are of this form. 
	However, our results are equally valid for flat $R$-linear categories; the only changes which need to be made are cosmetic.
\end{remark}

\begin{lemma}
    The natural map $\Mod_R RC\rightarrow\Mod_R C$ is an equivalence for any small category
    $C$. Moreover, this equivalence is compatible with the actions described above.
\end{lemma}

\begin{proof}
    This follows from the standard $2$-adjunction relating categories and $R$-linear categories, see for instance \cite{Kelly}*{Chapter~2.5}.
\end{proof}

\begin{lemma}
    Given a morphism of commutative rings $R\stackrel{\phi}{\rightarrow} S$ the natural base change functor 
    \begin{equation*}
        \phi^*:\Mod_R RC\rightarrow\Mod_S SC
    \end{equation*}
    has a right adjoint $\phi_*$.
\end{lemma}

\begin{proof}
    The functor $\phi^*$ is given by applying $S\otimes_R-$ objectwise and $\phi_*$ is induced by 
    restriction of scalars. This is again induced by a standard $2$-adjunction between $R$-linear and $S$-linear 
    categories corresponding to $\phi$.
\end{proof}

\section{Generalities on derived categories of small categories over a commutative ring}

Again $R$ is a fixed commutative ring which we now also assume is noetherian, and $C$ is a small category with $R$-linearisation $RC$.
The (unbounded) derived category $\Drm(RC)$ of $RC$ is the category of
complexes of right $RC$-modules where quasi-isomorphisms have been inverted. We note that this is 
a compactly generated triangulated category and the compact objects are, up to quasi-isomorphism, precisely the bounded 
complexes of projective $RC$-modules.

Recall that a localizing subcategory of $\Drm(RC)$ is a full triangulated subcategory of $\Drm(RC)$ closed under
coproducts (any such subcategory is automatically closed under direct summands). We want to consider to what
extent the localizing subcategories of $\Drm(RC)$ are determined by the localizing subcategories
of $\Drm(k(p)C)$ as $p$ ranges over the prime ideals of $R$. This is inspired by work of Neeman \cite{neeman-chromatic} 
who showed that in the case $C$ is the terminal category i.e., $RC = R$, the localizing subcategories 
of $\Drm(R)$ are determined by those of the $\Drm(k(p))$. We restrict to noetherian rings as, 
even in the case $RC=R$,
it is known that $\Spec R$ does not determine the localizing subcategories of $\Drm(R)$ in general.

Let us begin with the observation that the action of $\Modu R$ on $\Modu_R C$ can be derived.

\begin{lemma}
The bifunctor $\Modu R \times \Modu_RC \rightarrow \Modu_RC$ is left balanced, with respect to 
flat $R$-modules and objectwise $R$-flat $RC$-modules, i.e.\ it is exact when either the first variable is flat or 
the second variable is objectwise flat. It admits a left derived functor, independent up to isomorphism of which variable it is derived in, 
which gives a left action $\Drm(R)\times \Drm(RC) \rightarrow \Drm(RC)$ in the sense of \cite{stevenson-support}.
\end{lemma}
\begin{proof}
Given $F\in \Modu_RC$ such that $F$ is objectwise $R$-flat it is clear $-\otimes_RF$ is exact. 
As $\Modu_RC$ has enough projectives,
and the projective $RC$-modules are componentwise projective we see $\Modu_RC$ 
has enough objectwise $R$-flat modules. It is thus clear the functor can be left derived, using resolutions either in 
$\Modu R$ or $\Modu_RC$, and that it does not matter, up to quasi-isomorphism, on which side the resolution is taken 
(i.e., $-\otimes_R-$ is balanced as claimed).
It is straightforward to check this gives an associative and unital action of $\Drm(R)$ on $\Drm(RC)$.
\end{proof}

\begin{remark}
Given $E\in \Drm(R)$ and $F\in \Drm(RC)$ we will simply denote $E\otimes_R^\Lrm F$ by $E\otimes_R F$ or even $E\otimes F$; no confusion 
should result as we will almost exclusively work with derived functors (frequently with $R$ fixed or clear from the context).
\end{remark}

This allows us to utilize the machinery of tensor actions to analyze localizing subcategories of $\Drm(RC)$. 
After giving a convenient lemma and some notation we will recall the main result from this theory that we will need.

\begin{lemma}\label{lem:tensorclosure}
    Any localizing subcategory $\Lrm\subseteq\Drm(RC)$ is closed under tensoring with
    complexes of $R$-modules. Explicitly, for any $M\in \Drm(R)$ and $X\in \Lrm$ we have $M\otimes_RX\in \Lrm$.
    \begin{proof}
        Evidently, if $X\in\Lrm$, then $R\otimes_RX\we X\in\Lrm$. Since $-\otimes_RX$
        preserves coproducts, it follows that the subcategory of $\Drm(R)$
        consisting of complexes of $R$-modules $M$ such that $M\otimes_RX\in\Lrm$ is
        localizing and contains
        $R$. Since $R$ is a compact generator of $\Drm(R)$, the lemma follows.
    \end{proof}
\end{lemma}

Let $f$ be an element of $R$. We denote by $K_\infty(f)$ the \emph{stable Koszul complex of} $f$
\begin{displaymath}
R\rightarrow R_f
\end{displaymath}
where the map is the canonical one. Given a prime ideal $p$ of $R$ we set
\begin{displaymath}
K_\infty(p) = K_\infty(f_1)\otimes_R \cdots \otimes_R K_\infty(f_n),
\end{displaymath}
where $f_1,\ldots,f_n$ is a choice of generators for $p$. The resulting complex is independent 
of the choice of generators up to quasi-isomorphism (independence is usually left as an exercise but a proof can be found for instance in \cite{Greenlees}*{Lemma~2.3}).

Given $p \in \Spec R$ we define the object $\Gamma_pR$ to be $K_\infty(p)\otimes_R R_p$.
We recall from \cite{stevenson-support} 
that $\Gamma_pR\otimes_R\Gamma_pR \we \Gamma_pR$ and for $p\neq q$ in $\Spec
R$ we have $\Gamma_pR\otimes_R\Gamma_qR = 0$.

\begin{remark}
In more familiar language, the object $K_\infty(p)$ corresponds to taking local cohomology with support in $V(p)$ in the sense that the local cohomology functor is isomorphic to $K_\infty(p) \otimes (-)$. Thus $\Gamma_p R$ can be thought of as corresponding to ``$p$-localized local cohomology on $V(p)$.'' In general it differs from the residue field $k(p)$, which is rarely tensor idempotent. In certain situations, for instance if $R = \ZZ$, one can express $\Gamma_p R$ as a desuspension of a flat resolution of $E(k(p))$, the injective envelope of the residue field at $p$; for instance given a prime $p\in \ZZ$ one has $\Gamma_{(p)}\ZZ \cong \Sigma^{-1} E(\ZZ/p\ZZ)$. However, in general the precise relationship between $\Gamma_p R$, $k(p)$, and $E(k(p))$ seems to be more subtle.
\end{remark}

As a final point of notation, we will use $\langle S\rangle$ to denote the smallest localizing subcategory of a triangulated category generated by some collection of objects $S$.

\begin{theorem}[\cite{stevenson-support}*{Theorem~6.9}]\label{thm:ltg}
Given an object $X$ of $\Drm(RC)$ there is an equality of localizing subcategories
\begin{displaymath}
\langle X \rangle = \langle \Gamma_pR\otimes_R X\; \vert \; p\in \Spec R\rangle.
\end{displaymath}
It follows that $\Gamma_pR\otimes_R X \we 0$ for all prime ideals $p$ if and only if $X\we 0$.
\end{theorem}

\begin{corollary}
    If $X\in\Drm(RC)$ is non-zero, then there is some prime ideal $p$ of $R$ such that $k(p)\otimes_RX$ is
    not zero.
    \begin{proof}
        By the theorem there is a $p$ such that $\Gamma_pR\otimes_RX$ is non-zero. The result now follows as 
        $\langle \Gamma_pR\rangle = \langle k(p)\rangle$ in $\Drm(R)$
        by Neeman~\cite{neeman-chromatic}*{Section 2}, which implies $k(p)\otimes_RX \we 0$ if and only 
        if $\Gamma_pR\otimes_R X\we 0$.
    \end{proof}
\end{corollary}


We now turn to analyzing the localizing subcategories of $\Drm(RC)$ in terms of the `fibres' 
$\Drm(k(p)C)$ for $p\in \Spec R$. 
Let $\Lscr$ be the class defined in the following way. It comes equipped with a surjective map
$\Lscr\xrightarrow{s}\Spec R$, and the fiber over $p\in\Spec R$ is the class of
localizing subcategories of $\Drm(k(p)C)$. We will define a pair of maps
\begin{equation*}
    \xymatrix{
        \left\{\text{localizing subcategories $\Lrm$ of $\Drm(RC)$}\right\} \ar@<1ex>[r]^-f    &
        \left\{\text{sections $l$ of $\Lscr\xrightarrow{s}\Spec R$}\right\}\ar@<1ex>[l]^-g
    }.
\end{equation*}
In order to define the maps in the most convenient manner we require a little preparation.

\begin{lemma}\label{lem:basechangesplits}
	If $X$ is in the image of the forgetful functor $\Drm(k(p)C)\rightarrow \Drm(RC)$ then $k(p)\otimes_RX$ is 
	a direct sum of suspensions of $X$. In particular, the base change functor 
	$\Drm(RC) \rightarrow\Drm(k(p)C)$ is essentially surjective up to summands.
	\begin{proof}
			Let $X$ be as in the statement i.e., $X$ is a complex of $k(p)C$-modules regarded as a 
			complex of $RC$-modules. Then
			\begin{equation*}
				k(p)\otimes_RX\we \left(k(p)\otimes_Rk(p)\right)\otimes_{k(p)}X
			\end{equation*}
			is a coproduct of suspensions of $X$ since $k(p)\otimes_Rk(p)$ is a coproduct of suspensions of $k(p)$. 
			As the base change functor $\Drm(RC) \rightarrow\Drm(k(p)C)$ is just $k(p)\otimes_R -$ the final statement of the lemma is an immediate consequence.
	\end{proof}
\end{lemma}

\begin{lemma}\label{lem:fdefined}
	Let $\Lrm$ be a localizing subcategory of $\Drm(RC)$. Then $\add(k(p)\otimes_R\Lrm)$, the closure of $k(p)\otimes_R\Lrm$ under
	summands and isomorphisms in $\Drm(k(p)C)$, is a localizing subcategory of $\Drm(k(p)C)$.
	\begin{proof}
		It is evident $\add(k(p)\otimes_R\Lrm)$ is closed under suspensions and coproducts in $\Drm(k(p)C)$ as derived base change 
		is exact and coproduct preserving. Thus it is sufficient to show $\add(k(p)\otimes_R\Lrm)$ is closed under triangles. Suppose 
		$X\rightarrow Y\rightarrow Z\rightarrow$ is a triangle with $X,Y\in \add(k(p)\otimes_R\Lrm)$. Without loss of generality we may 
		assume $X,Y\in k(p)\otimes_R\Lrm$. By Lemma~\ref{lem:tensorclosure} the restrictions of $X$ and $Y$ lie in $\Lrm$, 
		so we deduce the restriction of $Z$ lies in $\Lrm$. Hence $k(p)\otimes_RZ$ is in $k(p)\otimes_R\Lrm$ and 
		using Lemma~\ref{lem:basechangesplits} we see $Z$ is in $\add(k(p)\otimes_R\Lrm)$ proving the lemma.
	\end{proof}
\end{lemma}

The function $f$ is defined as follows: we set
$f(\Lrm)(p)=\add(k(p)\otimes_R\Lrm)$ which is localizing by Lemma~\ref{lem:fdefined}. Given a section $l$ of $s$, define $g(l)$ as the localizing
subcategory
\begin{equation*}
    \left\{X\in\Drm(RC)\;|\;\text{$k(p)\otimes_RX\in l(p)$ for all primes $p\in\Spec R$}\right\}.
\end{equation*}
There is another natural function
\begin{equation*}
    \xymatrix{
        \left\{\text{localizing subcategories $\Lrm$ of $\Drm(RC)$}\right\}    &
        \left\{\text{sections $l$ of $\Lscr\xrightarrow{s}\Spec
    R$}\right\}\ar[l]^-{g'}
    }
\end{equation*}
defined as follows:
let $g'$ be the function that
takes $l$ to the localizing subcategory generated by the objects $X$ of $l(p)$ for all $p$,
viewed as $RC$-modules in the natural way, i.e.\
\begin{displaymath}
g'(l) = \langle l(p) \; \vert \; p\in \Spec R\rangle.
\end{displaymath}

\begin{lemma}\label{lem:gfcontainment}
    If $\Lrm$ is a localizing subcategory of $\Drm(RC)$, then $g'(f(\Lrm))\subseteq\Lrm\subseteq g(f(\Lrm))$.
    \begin{proof}
        The inclusion $\Lrm\subseteq g(f(\Lrm))$ is clear:
				\begin{displaymath}
				g(f(\Lrm)) = \{X\in \Drm(RC)\; \vert \; k(p)\otimes_R X \in \add(k(p)\otimes_R \Lrm) \; \forall\: p\in \Spec R\} \supseteq \Lrm.
				\end{displaymath}
				To show the other inclusion, note that $g'(f(\Lrm))$ is generated by
        $k(p)\otimes_RX$ as $X$ ranges over the objects of $\Lrm$ and $p$ ranges over the
        primes of $R$. But, by Lemma~\ref{lem:tensorclosure}, these are all in $\Lrm$.
    \end{proof}
\end{lemma}

\begin{lemma}\label{lem:splitinjection}
    Suppose that $l$ is a section of $s$. Then, $f(g'(l))=l=f(g(l))$. In particular, $f$ is
    surjective.
    \begin{proof}
        The value of $f(g'(l))$ at a prime $p$ consists of the localizing subcategory of
        $\Drm(k(p)C)$ generated by the complexes $k(p)\otimes_RX$
        for $X\in l(p)$. By Lemma~\ref{lem:basechangesplits} $k(p)\otimes_RX$ 
        is a direct sum of suspensions of $X$ and thus $f(g'(l))=l$. Similarly $l=f(g(l))$, 
        proving the lemma.
    \end{proof}
\end{lemma}

Our goal is to show that $g'(f(\Lrm))=\Lrm=g(f(\Lrm))$. This will prove that $g$ and $f$ are
inverse bijections and so gives a description of the lattice of localizing subcategories 
of $\Drm(RC)$ in terms of the corresponding derived categories over the residue fields 
of $\Spec R$.

\section{Proof of the main theorem}
This section is dedicated to proving $g'(f(L))=L=g(f(L))$.

Write $\Gamma_p\Drm(RC)$ for the localizing subcategory consisting of objects $X$ supported
at $p\in\Spec R$ i.e., those $X$ satisfying $k(q)\otimes_RX \we 0$ for $q\neq p$. Equivalently, 
one can describe $\Gamma_p\Drm(RC)$ as the essential image of $\Gamma_pR\otimes_R-$ in $\Drm(RC)$. 
We can restrict $f$ to the class of localizing subcategories of $\Gamma_p\Drm(RC)$.

\begin{proposition}\label{prop:main}
    The following are equivalent:
    \begin{enumerate}
        \item   the functions $f$ and $g$ are inverse bijections;
        \item   the restrictions $f_p$
            and $g_p$
            \begin{equation*}
                \xymatrix{
                    \left\{\text{localizing subcategories of $\Gamma_p\Drm(RC)$}\right\}
                    \ar@<1ex>[r]^-{f_p}    &
                    \left\{\text{localizing subcategories of
                    $\Drm(k(p)C)$}\right\}\ar@<1ex>[l]^-{g_p}
                }
            \end{equation*}
            are inverse bijections for all primes $p$;
        \item   for every prime ideal $p$ in $\Spec R$ and for every object $X$ of $\Gamma_p\Drm(RC)$, the localizing subcategories
            $\langle k(p)\otimes_RX\rangle$ and $\langle X\rangle$ are the same.
    \end{enumerate}
    \begin{proof}
        Clearly (1) implies (2). That (2) implies (3) follows from the fact that
        the localizing subcategories $\langle X\rangle$ and $\langle k(p)\otimes_RX\rangle$
        have the same image under $f_p$. Since $f$ is surjective, to prove that (3) implies
        (1), it suffices to prove that (3) implies $f$ is injective. Assuming this for a moment,
        Lemma~\ref{lem:splitinjection} says that both $g$ and $g'$ are inverses for $f$,
        which must then coincide.

        Assume now that $\Lrm$ is a localizing subcategory of $\Drm(RC)$ and that $X\in\Lrm$.
        It suffices to show that $X\in g'(f(\Lrm))$ since we have the other containment by Lemma~\ref{lem:gfcontainment}. 
				Under the assumption (3), $\Gamma_pR\otimes_RX\in g'(f(\Lrm))$ for every prime ideal
        $p$ in $\Spec R$ because $k(p)\otimes_R\Gamma_pR\otimes_RX\iso k(p)\otimes_RX$. Hence there is a containment of 
				localizing subcategories
				\begin{displaymath}
				\langle \Gamma_pR\otimes_RX\; \vert \; p\in \Spec R\rangle \subseteq g'(f(\Lrm)).
				\end{displaymath}
        By Theorem~\ref{thm:ltg}, $X$ lies in $\langle \Gamma_pR\otimes_RX\;
        \vert \; p\in \Spec R\rangle$, so $X\in g'(f(\Lrm))$ completing the proof.
    \end{proof}
\end{proposition}

The following observation is our main `theorem'.

\begin{theorem}\label{thm:main}
Let $p$ be a prime ideal of $R$ and $X$ an object of $\Gamma_p\Drm(RC)$. Then $X\in \langle k(p)\otimes_RX\rangle$ and hence
\begin{displaymath}
\langle k(p)\otimes_R X\rangle = \langle X \rangle.
\end{displaymath}
	\begin{proof}
		Let $X$ be as in the lemma and consider the following full subcategory of $\Drm(R)$
		\begin{displaymath}
		\Mrm = \{E\in \Drm(R)\; \vert \; E\otimes_R X \in \langle k(p)\otimes_R X\rangle\}.
		\end{displaymath}
		As $\langle k(p)\otimes_R X\rangle$ is a localizing subcategory it follows that $\Mrm$ is 
		also localizing (this is relatively straightforward but a proof can be found in \cite{stevenson-support}*{Lemma~3.8}). 
		It is immediate from the definition that $k(p)\in \Mrm$ and so $\langle k(p)\rangle \subseteq \Mrm$. By Neeman's 
		classification result \cite{neeman-chromatic} we have $\Gamma_pR \in \langle k(p)\rangle$ and hence $\Gamma_pR$ 
		also lies in $\Mrm$. Thus $\Gamma_pR\otimes_RX \in \langle k(p)\otimes_RX\rangle$ and it only remains to 
		observe that $X\in \Gamma_p\Drm(RC)$ implies $\Gamma_pR\otimes_RX \we X$.
	\end{proof}
\end{theorem}

\begin{corollary}\label{cor:main}
Let $R$ be a commutative noetherian ring and $C$ a small category. Then the assignments
\begin{equation*}
    \xymatrix{
        \left\{\text{localizing subcategories $\Lrm$ of $\Drm(RC)$}\right\} \ar@<1ex>[r]^-f    &
        \left\{\text{sections $l$ of $\Lscr\xrightarrow{s}\Spec R$}\right\}\ar@<1ex>[l]^-g
    }
\end{equation*}
are inverse to one another.
	\begin{proof}
		By Proposition~\ref{prop:main} it is sufficient to verify condition (3) i.e., that for every $X\in \Gamma_p\Drm(RC)$ 
		we have $X\in \langle k(p)\otimes_R X\rangle$. This is precisely the content of the theorem and so we see $f$ and $g$ 
		are inverse.
	\end{proof}
\end{corollary}

\begin{remark}
As noted in Remark~\ref{rem:lazy} our results are also valid in the case $D$ is a flat $R$-linear category and we consider 
$\Drm(\Mod_R D)$. One just needs to replace $k(p)C$ by $k(p)\otimes_R D$, the base change of $D$ to $k(p)$; the arguments 
don't change.
\end{remark}

\section{Dynkin quivers}

In this section we give a concrete application of the formalism above by considering the case that $C$ is the path category of a simply laced 
Dynkin quiver.
Let $Q$ be a quiver whose underlying graph is a simply laced Dynkin diagram.
We can naturally view $Q$ as a poset i.e., a small category and apply our result to the study of the derived category, $\Drm(RQ)$, of representations of $Q$ over $R$.
This yields the following extension of work of Ingalls and Thomas
\cite{ingalls-thomas}, where we refer the reader for information about
noncrossing partitions.

\begin{corollary}\label{cor:ADEclassification}
Let $R$ be a commutative noetherian ring, $Q$ a simply laced Dynkin quiver, and denote by $RQ$ the $R$-linear path algebra of $Q$. There is an 
isomorphism of lattices
\begin{displaymath}
\xymatrix{
        \left\{\text{localizing subcategories of $\Drm(RQ)$}\right\} \ar@<1ex>[r]^-f    &
        \left\{\text{functions $\Spec R \rightarrow \mathrm{NC}(Q)$}\right\}\ar@<1ex>[l]^-g
    },
\end{displaymath}
where $\mathrm{NC}(Q)$ denotes the lattice of noncrossing partitions associated to $Q$.
	\begin{proof}
		Corollary~\ref{cor:main} applies so it just remains to demonstrate there is a bijection
		\begin{displaymath}
			\{\text{sections of $\Lscr\stackrel{s}{\rightarrow} \Spec R$}\} \we \Hom(\Spec R, \mathrm{NC}(Q)).
		\end{displaymath}
		This follows from \cite{krause-report}*{Theorem~6.10} which shows, without restriction on the field $k$, that 
		there is a bijection between the lattice of thick subcategories of $\Drm^b(kQ)$ and $\mathrm{NC}(Q)$. As $kQ$ is hereditary 
		and of finite representation type, $\Drm(kQ)$ is pure-semisimple i.e., every object is a direct sum 
		of compact objects, and so we deduce a bijection between the lattice of localizing subcategories of $\Drm(kQ)$ 
		and $\mathrm{NC}(Q)$. Thus sections of $\Lscr\rightarrow \Spec R$ are nothing but functions from $\Spec R$ 
		to $\mathrm{NC}(Q)$.
	\end{proof}
\end{corollary}

\begin{remark}
One can also use Lemma~\ref{lem:splitinjection} and Krause's extension of a result by Igusa and Schiffler \cite{krause-report}*{Theorem~6.10} to 
get partial information on the lattice of localizing subcategories of $\Drm(RQ)$ for an arbitrary quiver $Q$.
\end{remark}

In this situation we can also obtain a classification of the thick subcategories of $\Drm^\perf(RQ)$, the category of perfect complexes of $RQ$-modules. Recall that $\Drm^\perf(RQ)$ is the full subcategory of $\Drm(RQ)$ consisting of those objects quasi-isomorphic to a bounded complex of finitely generated projective modules; it is a thick subcategory and is the subcategory of compact objects in $\Drm(RQ)$. As in the case of $\Drm^\perf(R)$, the thick subcategories of $\Drm^\perf(RQ)$ are given by a sublattice of the lattice of localizing subcategories defined by a certain specialization closure condition.

\begin{definition}
	We call a function $\sigma\colon \Spec R \to \mathrm{NC}(Q)$ \emph{specialization closed} if whenever $p\subseteq q$ we have $\sigma(p)\leq \sigma(q)$ in $\mathrm{NC}(Q)$. 
\end{definition}

\begin{remark}
	This recovers the usual notion of specialization closure of subsets of $\Spec R$ when $Q = A_1$ and so $\mathrm{NC}(Q) = \{0,1\}$. Moreover, returning to the general simply laced case, if $\Lrm$ is a localizing subcategory with $f(\Lrm)$ specialization closed then for $p\subseteq q$ we have
	\begin{displaymath}
		k(p)\otimes \Lrm \neq 0 \quad \Rightarrow \quad k(q)\otimes \Lrm \neq 0.
	\end{displaymath}
\end{remark}

We will show that specialization closed functions $\Spec R \to \mathrm{NC}(Q)$ classify smashing subcategories of $\Drm(RQ)$ and that the telescope conjecture holds. Combining these two results gives the claimed classification result for thick subcategories of $\Drm^\perf(RQ)$.
We begin by recalling a useful fact and then present the easiest part of the argument.

\begin{lemma}\label{lem:bill}
Let $p$ be a prime ideal of $R$ and let $M$ be an indecomposable $k(p)Q$-module with dimension vector $\alpha$. Then there is a rigid lattice $\widetilde{M}$ over $RQ$, i.e.\ $\widetilde{M}$ is $R$-free and $\Ext^1_{RQ}(\widetilde{M}, \widetilde{M}) = 0$, with rank vector $\alpha$. Moreover, for any $q\in \Spec R$ the module $k(q)\otimes \widetilde{M}$ is the unique indecomposable $k(q)Q$-module with dimension vector $\alpha$. In particular,
	\begin{displaymath}
		k(p)\otimes \widetilde{M} \cong M.
	\end{displaymath}
	\begin{proof}
		This is a (very) special case of a result of Crawley-Boevey \cite{crawley-boevey-rigid}*{Theorem~1}.
	\end{proof}
\end{lemma}

\begin{lemma}\label{lem:sccompact}
Let $\sigma\colon \Spec R \to\mathrm{NC}(Q)$ be specialization closed. Then the localizing subcategory $\Lrm = g(\sigma)$ is generated by objects of $\Drm^\perf(RQ)$. 
	\begin{proof}
	We prove this by just writing down a (rather redundant) generating set for $\Lrm$. For each prime ideal $p$ such that $k(p)\otimes \Lrm\neq 0$ let $M(p)$ be a compact generator for the localizing subcategory of $\Drm(k(p)Q)$ generated by $k(p)\otimes \Lrm$. Since $M(p)$ is a finite sum of (suspensions of) indecomposable modules in $\Drm(k(p)Q)$ we can lift it to a lattice $\widetilde{M(p)}$ in $\Drm(RQ)$ using Lemma~\ref{lem:bill}. In particular, it is easily seen that $\widetilde{M(p)}$ is compact in $\Drm(RQ)$. Set
	\begin{displaymath}
		G = \{K(p)\otimes \widetilde{M(p)} \; \vert \; p\in \Spec R \; \text{with}\; k(p)\otimes \Lrm \neq 0\} \quad \text{and} \quad \Lrm' = \langle G\rangle,
	\end{displaymath}
	where $K(p)$ denotes the Koszul complex for $p$ defined by
    \begin{equation*}
        K(p)=\bigotimes_{i=1}^r\mathrm{cone}(R\xrightarrow{f_i} R)
    \end{equation*}
    where $p$ is generated by $f_1,\ldots,f_r$. (Recall that this implicitly
    means the derived tensor product over $R$.)
    Since $K(p)\in \Drm^\perf(R)$ and $\widetilde{M(p)}\in \Drm^\perf(RQ)$, the set $G$ consists of compact objects by \cite{stevenson-support}*{Lemma~4.6}.
	
	For primes $p \subseteq q\in \Spec R$ the object $k(q) \otimes (K(p)\otimes \widetilde{M(p)})$ is a finite sum of suspensions of copies of the $k(q)Q$-module $k(q)\otimes \widetilde{M(p)}$. This latter module can be described as follows: each indecomposable summand of $M(p)$ corresponds to an indecomposable $k(q)Q$-module, namely the indecomposable with the same dimension vector, and $k(q)\otimes \widetilde{M(p)}$ is the corresponding sum of these indecomposable $k(q)Q$-modules. In particular, $M(p)$ and $k(q)\otimes \widetilde{M(p)}$ correspond to the same element of $\mathrm{NC}(Q)$. If, on the other hand, $p\nsubseteq q$ then $k(q) \otimes (K(p)\otimes \widetilde{M(p)}) = 0$.
	
	Putting everything together we see that
		\begin{align*}
			\langle k(q)\otimes \Lrm' \rangle &= \langle k(q) \otimes K(p)\otimes \widetilde{M(p)} \; \vert \; p\in \Spec R \; \text{with} \; k(p)\otimes \Lrm \neq 0\rangle \\
			&= \langle k(q)\otimes \widetilde{M(q)}\rangle \\
			&= \langle M(q) \rangle \\
			&= \langle k(q)\otimes \Lrm\rangle,
		\end{align*}
		where the second equality follows from the computation in the preceding paragraph together with specialization closure of $\sigma$, and the third and fourth equalities are by definition of $M(q)$ and $\widetilde{M(q)}$. This shows that $f(\Lrm) = f(\Lrm')$ and thus, by the classification of localizing subcategories, $\Lrm = \Lrm'$. We have thus exhibited a set of generators $G\subseteq \Drm^\perf(RQ)$ for $\Lrm$. 
	\end{proof}
\end{lemma}

We now continue with proving that the specialization closed functions $\Spec R \to
\mathrm{NC}(Q)$ classify smashing subcategories of $\Drm(RQ)$.
Combined with the above lemma this proves the telescope conjecture and classifies the thick subcategories of $\Drm^\perf(RQ)$.

Let us fix a smashing subcategory $\Srm$ of $\Drm(RQ)$, i.e.\ we have a localization sequence
	\begin{displaymath}
		\xymatrix{
		\Srm \ar[r]<0.5ex>^-{i_*} \ar@{<-}[r]<-0.5ex>_-{i^!} & \Drm(RQ) \ar[r]<0.5ex>^-{j^*} \ar@{<-}[r]<-0.5ex>_-{j_*} & \Srm^\perp
		}
	\end{displaymath}
where $i^!$ and $j_*$ are the right adjoints of the inclusion functors $i_*$
and the localization functor $j^*$ respectively and all of these functors preserve coproducts.
In particular, $\Srm^\perp$ is also a localizing subcategory of $\Drm(RQ)$.
In order to prove the result indicated above we start with two elementary
lemmas.

\begin{lemma}\label{lem:loccommutes}
Let $\Srm$ be as above. For any $Y\in \Drm(R)$ and $X\in \Drm(RQ)$ we have canonical isomorphisms
	\begin{displaymath}
		i_*i^!(Y\otimes X) \cong Y\otimes i_*i^!X \quad \text{and} \quad j_*j^*(Y\otimes X) \cong Y\otimes j_*j^*X.
	\end{displaymath}
		\begin{proof}
			Consider the localization triangle for $X$
				\begin{displaymath}
					i_*i^!X \to X \to j_*j^*X \to \Sigma i_*i^!X.
				\end{displaymath}
			Acting on this triangle with $Y$ gives a new triangle
				\begin{displaymath}
					Y\otimes i_*i^!X \to Y\otimes X \to Y\otimes j_*j^*X \to \Sigma(Y\otimes i_*i^!X).
				\end{displaymath}
			By Lemma~\ref{lem:tensorclosure} both $\Srm$ and $\Srm^\perp$ are closed under the $\Drm(R)$ action and so $Y\otimes i_*i^!X \in \Srm$ and $Y\otimes j_*j^*X \in \Srm^\perp$. The claimed isomorphisms follow immediately from the uniqueness of localization triangles.
		\end{proof}
\end{lemma}

\begin{lemma}\label{lem:map}
		Let $p'\in \Spec R$ and $M$ and $N$ be indecomposable $k(p')Q$-modules satisfying 
		\begin{displaymath}
			\Hom_{k(p')Q}(M,N)\neq 0
		\end{displaymath}
		and denote choices of their respective rigid lattice lifts by $\widetilde{M}$ and  $\widetilde{N}$. Then given $p\subseteq q\in \Spec R$ we have
		\begin{displaymath}
			\Hom_{RQ}(E(k(p))\otimes \widetilde{M}, E(k(q))\otimes \widetilde{N}) \neq 0,
		\end{displaymath}
		where $E(k(p))$ and $E(k(q))$ denote the injective envelopes of the residue fields $k(p)$ and $k(q)$.
	\begin{proof}
		We know there are rigid lattice lifts of $M$ and $N$ by Lemma~\ref{lem:bill}.
        We can choose, using the classification of indecomposable modules over $Q$, a non-zero $\phi\colon M\to N$ given on each component by
        matrices involving only zero and identity maps. It is then clear we can
        lift it to a non-zero $\widetilde{\phi}\colon \widetilde{M}\to
        \widetilde{N}$ such that $\widetilde{\phi}$, like $\phi$, is given
        componentwise by matrices whose only entries are zero and identity
        maps. On the other hand, since $p\subseteq q$, there is a non-zero map
        $\psi\colon E(k(p))\to E(k(q))$. It is thus evident by our choice of
        $\tilde{\phi}$ that either of the equal composites in the commutative square
			\begin{displaymath}
				\xymatrix{
					E(k(q))\otimes \widetilde{M} \ar[r]^-{1\otimes \widetilde{\phi}} & E(k(q))\otimes \widetilde{N} \\
					E(k(p))\otimes \widetilde{M} \ar[u]^-{\psi\otimes 1} \ar[r]_-{1\otimes \widetilde{\phi}} & E(k(p))\otimes \widetilde{N} \ar[u]_-{\psi\otimes 1}
				}
			\end{displaymath}
		gives the desired non-zero morphism.
	\end{proof}
\end{lemma}

Using this series of easy observations we can now dispose of the proof of the theorem in short order.

\begin{theorem}\label{thm:ADEsmashing}
Let $\Srm$ be a smashing subcategory of $\Drm(RQ)$ with notation as introduced above. Then $f(\Srm)\colon \Spec R \to \mathrm{NC}(Q)$ is specialization closed.
	\begin{proof}
		Fix $p\subseteq q \in \Spec R$ and an indecomposable module $M\in k(p)\otimes \Srm \subseteq \Drm(k(p)Q)$ with dimension vector $\alpha$. By Lemma~\ref{lem:bill} there is a lattice $\widetilde{M} \in \Drm^\perf(RQ)$ with $k(p)\otimes \widetilde{M} \cong M$ and $k(q)\otimes \widetilde{M}$ the unique indecomposable $k(q)Q$-module with dimension vector $\alpha$. We have to show that $k(q)\otimes \widetilde{M}$ is in $k(q)\otimes \Srm$. To this end consider the localization triangle
			\begin{displaymath}
				i_*i^!\widetilde{M} \to \widetilde{M} \to j_*j^*\widetilde{M} \to \Sigma i_*i^!\widetilde{M}.
			\end{displaymath}
		Pick an indecomposable summand $N$ of $k(q)\otimes j_*j^*\widetilde{M}$ and note that, by Lemma~\ref{lem:loccommutes}, $N\in \Srm^\perp$. We assume $N$ is non-zero as if $k(q)\otimes j_*j^*\widetilde{M}$ is zero then $k(q)\otimes \widetilde{M}$ is in $\Srm$ and we are done. Let $\widetilde{N}$ be a lattice lift of $N$. As we have assumed $k(q)\otimes j_*j^*\widetilde{M}$ is non-zero the morphism
			\begin{displaymath}
				\phi = k(q)\otimes \widetilde{M} \to k(q)\otimes j_*j^*\widetilde{M} \to N \cong k(q)\otimes \widetilde{N}
			\end{displaymath}
		must also be non-zero. Thus we can apply Lemma~\ref{lem:map} to produce a non-zero morphism
			\begin{displaymath}
				\gamma\colon E(k(p))\otimes \widetilde{M} \to E(k(q)) \otimes \widetilde{N}
			\end{displaymath}
		in $\Drm(RQ)$.
		
		On the other hand, by assumption $k(p)\otimes \widetilde{M}\in \Srm$ and $k(q)\otimes \widetilde{N}\in \Srm^\perp$. As both $\Srm$ and $\Srm^\perp$ are localizing, and for any prime ideal $p'$ we have $E(k(p'))\in \langle k(p')\rangle$, we see (as in the proof of Theorem~\ref{thm:main}) that
			\begin{displaymath}
				E(k(p))\otimes \widetilde{M} \in \Srm \quad \text{and} \quad E(k(q))\otimes \widetilde{N} \in \Srm^\perp.
			\end{displaymath}
		But this contradicts the existence of the non-zero morphism $\gamma$. Hence $N$ must have been zero, showing that $k(q)\otimes j_*j^*\widetilde{M} \cong 0$, which in turn implies (via Lemma~\ref{lem:loccommutes}) that $k(q)\otimes \widetilde{M} \in \Srm$ as desired.
	\end{proof}
\end{theorem}

This has the following, more palatable, consequences.

\begin{corollary}\label{cor:ADE1}
Let $R$ be a commutative noetherian ring and $Q$ a simply laced Dynkin quiver. Then $\Drm(RQ)$ satisfies the telescope conjecture: every smashing subcategory is generated by objects of $\Drm^\perf(RQ)$.
	\begin{proof}
		Suppose $\Srm$ is a smashing subcategory. Then by the classification given in Corollary~\ref{cor:ADEclassification} we know
			\begin{displaymath}
				\Srm = gf(\Srm).
			\end{displaymath}
		By Theorem~\ref{thm:ADEsmashing} the function $f(\Srm)$ is specialization closed and so by Lemma~\ref{lem:sccompact} we see $\Srm = gf(\Srm)$ is generated by objects of $\Drm^\perf(RQ)$ as claimed.
	\end{proof}
\end{corollary}

\begin{corollary}\label{cor:ADE2}
Let $R$ be a commutative noetherian ring and $Q$ a simply laced Dynkin quiver. There is an 
isomorphism of lattices
\begin{displaymath}
\xymatrix{
        \left\{\text{thick subcategories of $\Drm^\perf(RQ)$}\right\} \ar@<1ex>[r]^-f    &
        \left\{\text{specialization closed functions $\Spec R \rightarrow \mathrm{NC}(Q)$}\right\}\ar@<1ex>[l]^-g
    },
\end{displaymath}
where $\mathrm{NC}(Q)$ denotes the lattice of noncrossing partitions associated to $Q$.
	\begin{proof}
		Considering the classification of Corollary~\ref{cor:ADEclassification} and putting together Theorem~\ref{thm:ADEsmashing} and Lemma~\ref{lem:sccompact} gives a classification of the smashing subcategories of $\Drm(RQ)$ in terms of the specialization closed functions $\Spec R \to \mathrm{NC}(Q)$. By the previous corollary this is also the classification of the localizing subcategories of $\Drm(RQ)$ generated by objects of $\Drm^\perf(RQ)$. One obtains the isomorphism we have asserted in the statement in the standard way: by Thomason's localization theorem (see for example \cite{neeman-1996}*{Theorem~2.1}) the thick subcategories of $\Drm^\perf(RQ)$ are in order-preserving bijection with the localizing subcategories of $\Drm(RQ)$ which are generated by perfect complexes.
	\end{proof}
\end{corollary}

\begin{example}
    Let $R$ be a local $1$-dimensional domain. So, $\Spec R$ consists of two
    points: a generic point $\eta$ and a closed point $x$. We will consider the
    case of $Q=A_2$ in the above examples. The lattice $\mathrm{NC}(A_2)$ consists of
    the noncrossing partitions of the set $\{1,2,3\}$. A noncrossing partition
    of a cyclically ordered set $S$ determined by an equivalence relation $\sim$ is one
    where $x<y<z<w$, $x\sim z$, and $y\sim w$ together imply that $x\sim y\sim
    z\sim w$.

    Figure~\ref{fig:nc3} shows the lattice $\mathrm{NC}(A_2)$, the lattice of
    noncrossing partitions of $\{1,2,3\}$. We display each partition as
    determined by its largest equivalence classes. The class of all
    localizing subcategories of $\Drm(RA_2)$ in this case is simply two copies
    of the lattice below, indexed on $\eta$ and $x$. Figure~\ref{fig:Rnc3}
    shows the lattice of specialization closed functions $\Spec
    R\rightarrow\mathrm{NC}(A_2)$, which by the results above is the lattice of
    thick subcategories of $\Drm^{\perf}(RA_2)$.

    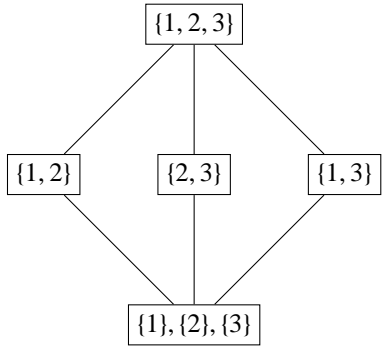
\begin{figure}[h]
        \centering
        \begin{tikzpicture}
            \node[draw,rectangle] (a) at (2,4) {$\prn{1,2,3}$};
            \node[draw,rectangle] (b) at (0,2) {$\prn{1,2}$};
            \node[draw,rectangle] (c) at (2,2) {$\prn{2,3}$};
            \node[draw,rectangle] (d) at (4,2) {$\prn{1,3}$};
            \node[draw,rectangle] (e) at (2,0) {$\prn{1},\prn{2},\prn{3}$};

            \draw[-]   (a) -- (b)
                        (a) -- (c)
                        (a) -- (d)
                        (b) -- (e)
                        (c) -- (e)
                        (d) -- (e);
        \end{tikzpicture}
        \caption{The lattice of noncrossing partitions of $\{1,2,3\}$. The
    coarser partitions are decreed to be bigger in the lattice structure.}
        \label{fig:nc3}
    \end{figure}

    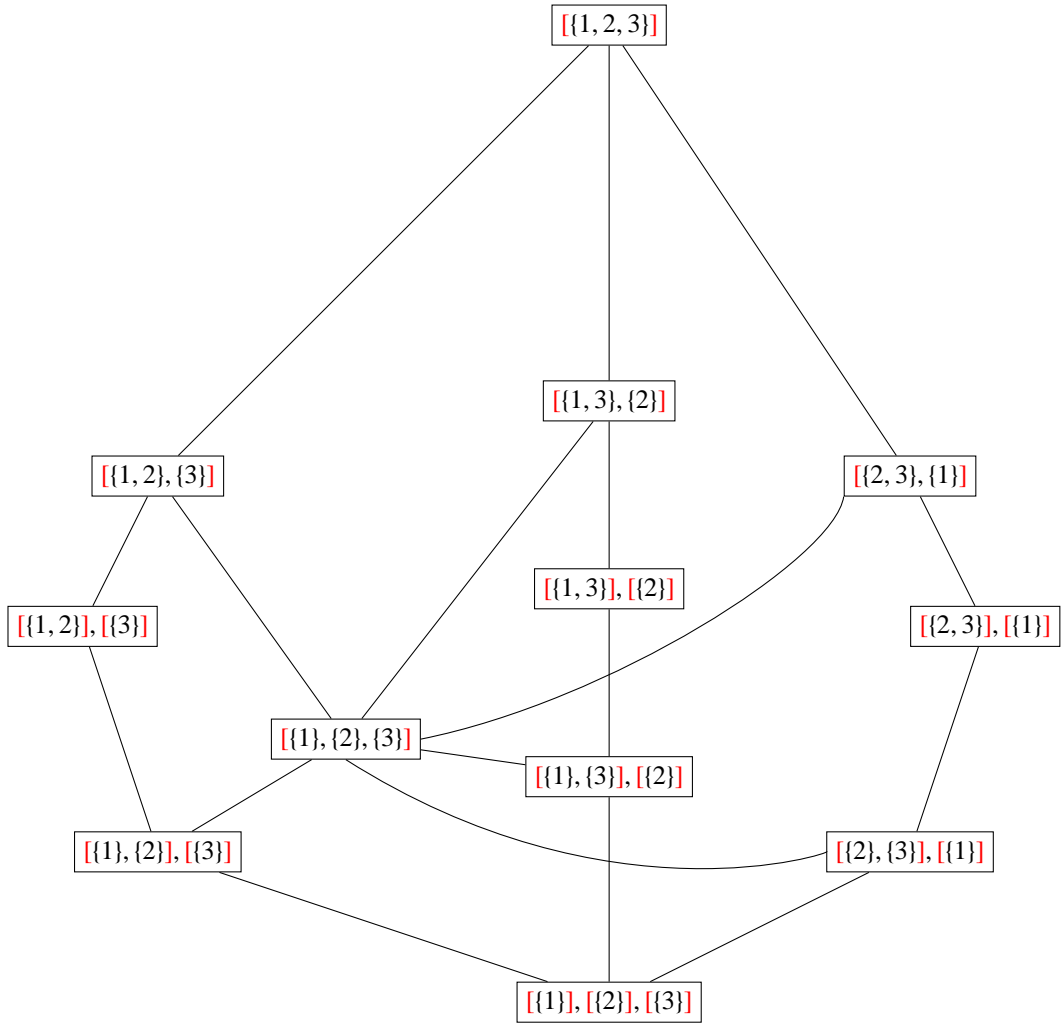
\begin{figure}[h]
        \centering
        \begin{tikzpicture}
            \node[draw,rectangle] (a) at (5,13) {$\rprn{\prn{1,2,3}}$};
            \node[draw,rectangle] (m) at (1.5,3.5) {$\rprn{\prn{1},\prn{2},\prn{3}}$};

            \node[draw,rectangle] (b) at (-1,7) {$\rprn{\prn{1,2},\prn{3}}$};
            \node[draw,rectangle] (c) at (-2,5) {$\rprn{\prn{1,2}},\rprn{\prn{3}}$};
            \node[draw,rectangle] (d) at (-1,2) {$\rprn{\prn{1},\prn{2}},\rprn{\prn{3}}$};

            \node[draw,rectangle] (e) at (5,8) {$\rprn{\prn{1,3},\prn{2}}$};
            \node[draw,rectangle] (f) at (5,5.5) {$\rprn{\prn{1,3}},\rprn{\prn{2}}$};
            \node[draw,rectangle] (g) at (5,3) {$\rprn{\prn{1},\prn{3}},\rprn{\prn{2}}$};
           
            \node[draw,rectangle] (h) at (9,7) {$\rprn{\prn{2,3},\prn{1}}$};
            \node[draw,rectangle] (i) at (10,5) {$\rprn{\prn{2,3}},\rprn{\prn{1}}$};
            \node[draw,rectangle] (j) at (9,2) {$\rprn{\prn{2},\prn{3}},\rprn{\prn{1}}$};

            \node[draw,rectangle] (k) at (5,0) {$\rprn{\prn{1}},\rprn{\prn{2}},\rprn{\prn{3}}$};

            \draw[-]    (a) -- (b)
                        (a) -- (e)
                        (a) -- (h)
                        (b) -- (c)
                        (c) -- (d)
                        (d) -- (k)
                        (e) -- (f)
                        (f) -- (g)
                        (g) -- (k)
                        (h) -- (i)
                        (i) -- (j)
                        (d) -- (m)
                        (g) -- (m)
                        (m) -- (b)
                        (m) -- (e)
                        (j) -- (k);

            \draw[-] (m.east) .. controls (5,4) and (8,5.9) .. (h.south west);
            \draw[-] (m.south) .. controls (5,1) and (8,2) .. (j.west);
            
        \end{tikzpicture}
        \caption{The lattice of specialization closed functions $\Spec
            R\rightarrow\mathrm{NC}(A_2)$ for $R$ a $1$-dimensional local
            domain. The partition given by the black parentheses is the
            noncrossing partition corresponding to the generic point $\eta$, while
            the partition determined by the red parentheses is the partition
            corresponding to the closed point $x$.}
        \label{fig:Rnc3}
    \end{figure}
\end{example}




\section{Towards telescopy}\label{sec:telescope}

 weahave seen in Corollary~\ref{cor:ADE1} the telescope conjecture holds for $\Drm(RQ)$ when $Q$ is an ADE quiver and $R$ is any commutative noetherian ring. Unfortunately we were not able to prove such a general statement for even arbitrary quivers, let alone arbitrary small categories. However, we do have some partial results and remarks which we present in this section which revolve around the following question.

\begin{question}\label{qn}
    Let $R$ be a noetherian commutative ring.
    Does the telescope conjecture hold for $\Drm(RC)$ when $C$ is an ordinary (not
    $R$-linear) category if it holds for $\Drm(k(p)C)$ for all $p\in\Spec R$?
\end{question}

We begin to answer this question by showing the bijection of Proposition~\ref{prop:main}(2) restricts to a bijection between the collections of smashing subcategories. Given a localizing subcategory $\Lrm$ of some triangulated category we will denote the associated acyclization and localization functors by $\Gamma_\Lrm$ and $L_\Lrm$ respectively.

\begin{remark}
	Throughout we will prove that some localizing subcategory $\Srm$ is smashing by exhibiting that the right orthogonal $\Srm^\perp$ is also localizing. 
	In order for this condition to be equivalent to $\Srm$ being smashing one needs to know the inclusion of $\Srm$ admits a right adjoint. In all 
	of the cases we consider $\Srm$ will clearly be generated by a set of objects, for instance it will be the localizing subcategory 
	generated by the image of some other smashing subcategory under an exact functor, and so the existence of the adjoint follows from Brown 
    representability. Indeed, in this case one has a generating set as any smashing subcategory of a compactly generated triangulated category has a set of generators by \cite{KrauseLoc}*{Theorem~7.4.1} and so one can apply Brown representability for well-generated categories as in \cite{NeeCat} (or see \cite{KrauseLoc}*{Theorem~5.1.1}). Thus we will suppress this part of the arguments throughout.
\end{remark}

Let us for the moment fix some $p\in \Spec R$ and denote by $i^*$ the functor $k(p)\otimes(-)\colon\Gamma_p\Drm(RC) \to \Drm(k(p)C)$ and 
its right adjoint by $i_*$.

\begin{lemma}\label{lem:smashingpres1}
	Suppose $\Srm$ is a smashing subcategory of $\Gamma_p\Drm(RC)$ and set
	\begin{displaymath}
		\Trm = f(\Srm) = \add( k(p)\otimes \Srm) \quad\text{and}\quad \Trm' = f(\Srm^\perp) = \add(k(p) \otimes \Srm^\perp).
	\end{displaymath}
	Then $\Trm'$ is the right orthogonal of $\Trm$ and hence $\Trm$ is a smashing subcategory of $\Drm(k(p)C)$.
	\begin{proof}
		If $X\in \Trm'$ then there is, by definition, some $\bar{X}\in \Srm^\perp$ such that $X$ is a summand of $i^*\bar{X}$. Given $Y\in \Trm$, 
		which we can assume to be of the form $i^*\bar{Y}$ with $\bar{Y}\in \Srm$, we have
		\begin{displaymath}
			\Hom(i^*\bar{Y}, i^*\bar{X}) \cong \Hom(\bar{Y}, i_*i^*\bar{X}).
		\end{displaymath}
		This latter hom-set is zero as $\bar{Y}\in \Srm$ and $i_*i^*\bar{X} \in \Srm^\perp$ by the closure of localizing subcategories under 
		the $\Drm(R)$ action. Thus $\Trm' \subseteq \Trm^\perp$.
		
		On the other hand, if $\Hom(i^*\Srm, Z) = 0$ for some $Z\in \Drm(k(p)C)$, then by adjunction $i_*Z\in \Srm^\perp$. Hence 
		$i^*i_*Z\in \Trm'$ and we know, by Lemma~\ref{lem:basechangesplits},
        $Z$ is a summand of $i^*i_*Z$. So $Z$ is in $\Trm'$, proving that $\Trm^\perp \subset \Trm'$ and completing the argument.
	\end{proof}
\end{lemma}

Now we fix a smashing subcategory $\Trm$ of $\Drm(k(p)C)$ and set
\begin{displaymath}
	\Srm = g(\Trm) = \langle i_*\Trm\rangle \quad\text{and}\quad \Srm' = g(\Trm^\perp) = \langle i_*\Trm^\perp\rangle.
\end{displaymath}
We wish to show $\Srm$ is smashing with right orthogonal $\Srm'$. We prove this
in the following four statements.

\begin{lemma}\label{lem:gen}
The subcategories $\Srm$ and $\Srm'$ generate $\Gamma_p\Drm(RC)$ i.e., we have an equality
\begin{displaymath}
	\langle \Srm\cup \Srm'\rangle = \Gamma_p\Drm(RC).
\end{displaymath}
	\begin{proof}
		Let $X$ be an object of $\Gamma_p\Drm(RC)$. By Theorem~\ref{thm:main} we know $X$ is in the localizing subcategory $\langle i_*i^*X\rangle$. 
		We have a localization triangle in $\Drm(k(p)C)$
		\begin{displaymath}
			\Gamma_\Trm i^*X \to i^*X \to L_\Trm i^*X \to\Sigma\Gamma_\Trm i^*X
		\end{displaymath}
		where $\Gamma_\Trm i^*X \in \Trm$ and $L_\Trm i^*X \in \Trm^\perp$. Applying $i_*$ gives a triangle in $\Drm(RC)$
		\begin{displaymath}
			i_*\Gamma_\Trm i^*X \to i_*i^*X \to i_*L_\Trm i^*X
            \to\Sigma i_*\Gamma_T i^*X
		\end{displaymath}
		with $i_*\Gamma_\Trm i^*X \in \Srm$ and $i_*L_\Trm i^*X \in \Srm'$ by definition. Thus $X\in \langle i_*i^*X\rangle \subseteq \langle \Srm\cup \Srm'\rangle$ 
		as claimed.
	\end{proof}
\end{lemma}

\begin{lemma}\label{lem:orthog}
There is a containment of triangulated subcategories $\Srm'\subseteq \Srm^\perp$.
	\begin{proof}
		It is enough to check that for every $t\in \Trm$ and $t'\in \Trm^\perp$ we have
		\begin{displaymath}
			\Hom(i_*t, i_*t') = 0.
		\end{displaymath}
		The required vanishing follows from the isomorphisms
		\begin{displaymath}
			\Hom(i_*t, i_*t') \cong \Hom(i^*i_*t, t') \cong \Hom(\coprod_\lambda \Sigma^{n_\lambda}t, t') \cong \prod_\lambda \Hom(\Sigma^{n_\lambda}t, t') = 0,
		\end{displaymath}
		where the first isomorphism is by adjunction, the second is by Lemma~\ref{lem:basechangesplits}, and the final hom-set vanishes by assumption.
	\end{proof}
\end{lemma}

\begin{lemma}\label{lem:verdierproduct}
There is an equality
\begin{displaymath}
\Gamma_p\Drm(RC) = \{X\in \Gamma_p\Drm(RC) \; \vert \; \exists \; \text{a
triangle} \; X'\to X\to X'' \to\Sigma X' \; \text{with} \; X'\in \Srm \; \text{and} \; X''\in \Srm'\}.
\end{displaymath}
	\begin{proof}
		It is routine to verify that the full subcategory defined on the right hand side above is localizing and it contains $\Srm$ and $\Srm'$ by definition. The equality then follows from Lemma~\ref{lem:gen}.
	\end{proof}
\end{lemma}

\begin{proposition}\label{prop:smashingpres2}
The subcategory $\Srm$ is smashing in $\Gamma_p\Drm(RC)$ with right orthogonal $\Srm'$.
	\begin{proof}
		We already know by Lemma~\ref{lem:orthog} that $\Srm' \subseteq \Srm^\perp$. Let $X$ be an object of $\Srm^\perp$. By the last lemma we know there is a triangle
		\begin{displaymath}
			X' \to X \to X'' \to\Sigma X'
		\end{displaymath}
		with $X'\in \Srm$ and $X''\in \Srm'$. But, since $X\in \Srm^\perp$ the first map must vanish implying $X'' \cong X\oplus \Sigma X'$. This in turn implies $X' \cong 0$ since $\Srm\cap \Srm' = 0$. We thus conclude that $X\cong X''$, i.e.\ $X\in \Srm'$ proving $\Srm^\perp = \Srm'$. In particular, $\Srm$ is smashing. 
	\end{proof}
\end{proposition}

We now have enough to prove that we can describe the smashing subcategories of $\Gamma_p\Drm(RC)$ in terms of the smashing subcategories of $\Drm(k(p)C)$.

\begin{theorem}\label{thm:smashingbijection}
There is an order preserving bijection
   \begin{equation*}
      \xymatrix{
          \left\{\text{smashing subcategories of $\Gamma_p\Drm(RC)$}\right\}
             \ar@<1ex>[r]^-{f_p}    &
             \left\{\text{smashing subcategories of
              $\Drm(k(p)C)$}\right\}\ar@<1ex>[l]^-{g_p}
               }
		\end{equation*}
	\begin{proof}
		We know from Proposition~\ref{prop:main}(2) that there is a bijection between the sets of localizing subcategories of $\Gamma_p\Drm(RC)$ and $\Drm(k(p)C)$ given by $f_p$ and $g_p$. By Lemma~\ref{lem:smashingpres1} and Proposition~\ref{prop:smashingpres2} both $f_p$ and $g_p$ send smashing subcategories to smashing subcategories and so the bijection restricts as claimed.
	\end{proof}
\end{theorem}

Obtaining the corresponding result for localizing subcategories generated by compact objects of $\Gamma_p\Drm(RC)$ and $\Drm(k(p)C)$ seems more subtle.
However, if $R$ is sufficiently nice at the prime ideal $p$ this is possible.
In order to state the result we need a simple preparatory lemma.

\begin{lemma}\label{lem:cg}
Let $p$ be a prime ideal of $\Spec R$. The category $\Gamma_p\Drm(RC)$ is a compactly generated triangulated category.
	\begin{proof}
		Recall that $\Gamma_p\Drm(RC)$ is the essential image of acting by
		\begin{displaymath}
			\Gamma_pR = K_\infty(p)\otimes_R R_p.
		\end{displaymath}
		It is clear that $\Drm(R_pC)$, the essential image of acting by $R_p$,
        is a compactly generated triangulated category. By \cite{stevenson-support}*{Corollary~4.11}
        the essential image of $K_\infty(p)_p\otimes_{R_p}(-)$ acting on $\Drm(R_pC)$, namely $\Gamma_p\Drm(RC)$, is also compactly generated (even by objects of $\Drm^\perf(R_pC)$).
	\end{proof}
\end{lemma}

In the statement and proof of the following proposition, the notation
$(\Gamma_p\Drm(RC))^c$ denotes the full subcategory of compact objects of
$\Gamma_p\Drm(RC)$.

\begin{proposition}\label{prop:cg-reg}
Let $p$ be a prime ideal of $R$ such that $R_p$ is regular. Then the assignments $f_p$ and $g_p$ of Proposition~\ref{prop:main}(2) induce an order preserving bijection between localizing subcategories of $\Gamma_p\Drm(RC)$ generated by objects of $(\Gamma_p\Drm(RC))^c$ and localizing subcategories of $\Drm(k(p)C)$ generated by objects of $\Drm^\perf(k(p)C)$.
	\begin{proof}
		The base change functor $\Gamma_p\Drm(RC) \to \Drm(k(p)C)$ has a coproduct preserving right adjoint and so sends compacts to compacts by \cite{neeman-1996}*{Theorem~5.1}. Thus it is clear that $f_p$ sends any localizing subcategory of $\Gamma_p\Drm(RC)$ generated by objects of $(\Gamma_p\Drm(RC))^c$ to a localizing subcategory generated by objects of $\Drm^\perf(k(p)C)$. The argument for $g_p$ is similar, using the fact that as $R_p$ is regular, the residue field $k(p)$ is compact, and so the right adjoint of the restriction functor $\Hom_R(k(p),-)$ is also coproduct preserving. 
	\end{proof}
\end{proposition}

As an immediate consequence of the theorem and the proposition we deduce the following corollary.

\begin{corollary}\label{cor:teletransfer}
Suppose $R_p$ is regular. Then $\Gamma_p\Drm(RC)$ satisfies the telescope conjecture if and only if $\Drm(k(p)C)$ satisfies the telescope conjecture.
	\begin{proof}
		Suppose $\Drm(k(p)C)$ satisfies the telescope conjecture and let $\Srm$ be a smashing subcategory of $\Gamma_p\Drm(RC)$. Then $f_p(S)$ is smashing in $\Drm(k(p)C)$ by Theorem~\ref{thm:smashingbijection} and $g_pf_p(\Srm) = \Srm$. Since we have assumed the telescope conjecture for $\Drm(k(p)C)$ we know $f_p(S)$ is generated by objects of $\Drm^\perf(k(p)C)$. Applying Proposition~\ref{prop:cg-reg} we deduce that $\Srm = g_pf_p(\Srm)$ is generated by objects which are compact in $\Gamma_p\Drm(RC)$. Thus the telescope conjecture holds for $\Gamma_p\Drm(RC)$.
        The other implication is clear since $i^*$ preserves compact objects.
	\end{proof}
\end{corollary}

This corollary already buys us something in a concrete setting, although it is not clear how to extend it to all of $\Drm(RC)$.

\begin{corollary}
Let $Q$ be a quiver and let $R$ be a commutative noetherian ring. For each $p\in \Spec R$ such that $R_p$ is regular the telescope conjecture holds for $\Gamma_p\Drm(RC)$.
	\begin{proof}
		By the previous corollary it is sufficient to verify the telescope conjecture for $\Drm(k(p)Q)$. This has been done by Krause and \stovicek, see \cite{krause-stovicek}*{Theorem~7.1}.
	\end{proof}
\end{corollary}

We give one additional lemma that could prove useful in resolving Question~\ref{qn}.

\begin{lemma}\label{lem:psmash}
If $\Srm$ is a smashing subcategory of $\Drm(RC)$ then for every $p\in \Spec R$ the localizing subcategory $\Gamma_p\Srm$ is smashing in $\Gamma_p\Drm(RC)$.
	\begin{proof}
		It is not hard to check that both $\Gamma_p\Srm$ and $\Gamma_p(\Srm^\perp)$ are localizing subcategories of $\Gamma_p\Drm(RC)$. Moreover,
		\begin{displaymath}
		\Gamma_p\Srm \subseteq \Srm \quad \text{and} \quad \Gamma_p(\Srm^\perp) \subseteq \Srm^\perp
		\end{displaymath}
		by Lemma~\ref{lem:tensorclosure}. In particular, $\Gamma_p(\Srm^\perp) \subseteq (\Gamma_p\Srm)^\perp$. Applying $\Gamma_pR\otimes_R(-)$ to localization triangles for $\Srm$ shows that every object $X$ of $\Gamma_p\Drm(RC)$ fits into a triangle
		\begin{displaymath}
			X' \to X \to X'' \to\Sigma X'
		\end{displaymath}
		with $X' \in \Gamma_p\Srm$ and $X'' \in \Gamma_p(\Srm^\perp)$ and so one can conclude the proof by arguing as in the proof of Proposition~\ref{prop:smashingpres2}.
	\end{proof}
\end{lemma}

In summary, we understand what happens at ``points'' and we can pass from a smashing subcategory of $\Drm(RC)$ to a smashing subcategory at each prime. What is not clear is how to use this pointwise information to deduce something about the original smashing subcategory. The naive idea, based on the existing proofs of the telescope conjecture in various instances, would be to prove some sort of specialization closure condition for the section corresponding to a smashing subcategory as in Theorem~\ref{thm:ADEsmashing}. One could then hope to combine such a condition with the fibrewise results above. However, the following example shows that one can not always expect specialization closure.

\begin{example}
    Consider the projection $\Spec k[x,y]\rightarrow\Spec k[x]$.
    We view $\Mod\, k[x,y]$ as a
    $k[x]$-linear category. This gives rise to an action of $\Drm(k[x])$ on $\Drm(k[x,y])$. Let $\Srm$ be the smashing subcategory of $\Drm(k[x,y])$
    determined by the closed curve $xy=1$. Then, the support of $\Srm$ with respect to the action of $\Drm(k[x])$ is open in $\Spec
    k[x]$. Of course, in this case the telescope conjecture does hold. 
\end{example}


%
%

\end{document}